\documentclass{amsart}
\usepackage{amsmath}
\usepackage{amsfonts}
\usepackage{amssymb,hyperref}
\usepackage[initials]{amsrefs}

\usepackage{color}

\newtheorem{thm}{Theorem}[section]
\newtheorem{cor}[thm]{Corollary}

\newtheorem{lemma}[thm]{Lemma}
\newtheorem{prop}[thm]{Proposition}

\theoremstyle{definition}
\newtheorem{defn}[thm]{Definition}
\newtheorem{remark}[thm]{Remark}

\newcommand{\bb}[1]{\mathbb{#1}}
\newcommand{\cl}[1]{\mathcal{#1}}
\newcommand{\ff}[1]{\mathfrak{#1}}

\newcommand{\cstar}{C^*(\bb F(n,2))}

\newcommand\Gammat{{\widetilde\Gamma}}

\begin{document}

\title{Non-closure of the set of quantum correlations via Graphs}
\author[K.~Dykema]{Ken Dykema$^*$}
\address{Department of Mathematics, Texas A \& M University, College Station, TX 77843-3368, USA}
\email{kdykema@math.tamu.edu}
\thanks{{}$^*$This work was supported by a grant from the Simons Foundation/SFARI (524187, K.D.)}

\author[V.~I.~Paulsen]{Vern I.~Paulsen$^\dag$}
\address{Institute for Quantum Computing and Department of Pure Mathematics, University of Waterloo,
Waterloo, ON, Canada N2L 3G1}
\email{vpaulsen@uwaterloo.ca}
\thanks{${}^\dag$Supported in part by NSERC}

\author[J.~Prakash]{Jitendra Prakash$^\dag$}
\address{Institute for Quantum Computing and Department of Pure Mathematics, University of Waterloo,
Waterloo, ON, Canada N2L 3G1}
\email{jprakash@uwaterloo.ca}


\begin{abstract}
We prove that the set of quantum correlations for a bipartite system of 5 inputs and 2 outputs is not closed.
Our proof relies on computing the correlation functions of a graph, which is a concept that we introduce. 

\end{abstract}
%\date{\today (Last revised)}

\maketitle

\section{Introduction}
Suppose that two labs, Alice's and Bob's, exist in an entangled state and each lab has a finite set of quantum experiments that they can perform and each experiment has a finite number of outcomes. The conditional probability that Alice gets outcome $a$ and Bob gets outcome $b$ given that they perform experiments $x$ and $y$ respectively, is denoted $p(a,b|x,y)$. Such densities are generally called {\it quantum correlations}. If we assume that each lab has $n$ experiments and each experiment has $k$ outcomes, then the set of all possible quantum correlations is a convex subset of $n^2k^2$-tuples. There are several, possibly different, mathematical models that could describe the elements in these sets.  The sets from the various models are denoted, $C_q(n,k), C_{qs}(n,k), C_{qa}(n,k)$ and $C_{qc}(n,k)$, and satisfy
\[ C_q(n,k) \subseteq C_{qs}(n,k) \subseteq C_{qa}(n,k) \subseteq C_{qc}(n,k).\] The {\it Tsirelson conjectures} \cites{Ts1, Ts2} are concerned with the relationships between the sets obtained by these various models.  Originally, it was not known if these sets were all the same or were possibly different.  A great deal of additional interest developed around these problems when it was shown that equality of two of these models, $C_{qa}(n,k) = C_{qc}(n,k)$ for all $n$ and $k$ was equivalent to the {\it Connes embedding conjecture} \cites{JNPPSW, Oz13, Fritz}, a major open problem in the theory of operator algebras.

Recently, Slofstra \cite{Sl17} has shown that the set of quantum correlations $C_q(n,k)$ is not closed, when the number of experiments and the number of outputs is sufficiently high ($n \sim 100, k =8$). Since $C_{qa}(n,k)$ is always closed, his result shows that $C_{q}(n,k) \ne C_{qa}(n,k),$ for some values of $n$ and $k$.  His proof relies on a number of deep constructions in geometric group theory, and the number $n$ is defined somewhat implicitly.  So it is natural to seek simpler proofs and to wonder about the case of small numbers of inputs and outputs. 

In this paper we will show that $C_q(5,2)$ is not closed and hence not equal to $C_{qa}(5,2)$, by studying the properties of a function that we call the {\it graph correlation function}.

Given a graph, we wish to study several functions that measure the least possible {\it total tracial correlation}, when we assign a projection of fixed trace to each vertex and measure the total correlation between projections that are at adjacent vertices.  The goals of this study are on the one hand to try and shed further light on the conjectures of Connes and Tsirelson and on the other hand to introduce this new parameter of a graph and show some of its connections to other problems. We will see that determining where this correlation function is equal to 0, is equivalent to finding the {\it fractional chromatic number} of the graph, when the algebra is abelian, and Man\v{c}inska-Roberson's {\it projective rank} \cites{Ro, Ro-thesis} of the graph when the algebra is required to be finite dimensional.

We begin with the definitions of the functions that we shall be interested in studying.

Let $G=(V,E)$ be a simple nonempty graph on $n$ vertices with vertex set $V$ and edges $E \subseteq V \times V$. If we let $\bb F(n,2)$ denote the free product of $n$ copies of the group of order 2, then  the full group $C^*$-algebra, $C^*(\bb F(n,2))$, is the universal unital $C^*$-algebra generated by projections, $e_v=e_v^2 = e_v^*$, $v \in V$.  By a {\it tracial state} on $\cstar$ we mean a positive unital linear functional, $\tau: \cstar \to \bb C$, satisfying $\tau(ab)=\tau(ba)$ for all $a,b\in C^*(\bb F(n,2))$. For $0 \le t \le 1$ we set 
\begin{multline}\label{f-qc-def}
f_{qc}(t) = \inf \bigg\{ \sum_{(v,w) \in E} \tau(e_ve_w): \tau \text{ is a tracial state on } C^*(\bb F(n,2)), \\ \tau(e_v)=t, \text{ for all } v\in V  \bigg\}.
\end{multline} Our notation suppresses the dependence of this function on the graph $G$. Notice that each edge $(v,w)$ appears twice in $E$ as $(v,w)$ and $(w,v)$. Thus if $|E|$ denotes the cardinality of the edge set of $G$, then it is twice the number of actual edges. 

Recall that every state $\tau$ on $\cstar$ has a {\it Gelfand-Naimark-Segal (GNS) representation}, that is, there exists a Hilbert space $\cl H$, a unital $*$-homomorphism $\pi: \cstar \to B(\cl H)$, and a unit vector $\psi \in \cl H$ such that  $\tau(a) = \langle \pi(a) \psi, \psi \rangle$ for all $a\in C^*(\bb F(n,2))$.  We shall call a state $\tau$ on $\cstar$ {\it finite dimensional} provided that the Hilbert space in the GNS representation is finite dimensional. We shall call a state {\it abelian} if
the image of $C^*(\bb F(n,2))$ under the GNS representation is commutative. This latter condition is equivalent to the existence of a probability space $(X, \mu)$ and measurable subsets $X_v$, such that $\tau(e_ve_w) = \mu(X_v \cap X_w)$, for all $v,w\in V$.

We set $f_q(t)$ (respectively, $f_{loc}(t)$) equal to the infimum in \eqref{f-qc-def} but taken over the set where $\tau$ is restricted to be a finite-dimensional (respectively, abelian) tracial state.

Here is the first relevance of this function.

\begin{prop}\label{ranks-and-functions}
Let $G$ be a graph on $n$ vertices.  Then,
\begin{enumerate}
\item $\big( \sup \{ t: f_{loc}(t) =0 \} \big)^{-1}$ is equal to the fractional chromatic number of $G$,
\item $\big( \sup \{ t: f_{q}(t)=0 \} \big)^{-1}$ is equal to Man\v{c}inska-Roberson's projective rank \cite{Ro} of $G$,
\item $\big( \sup \{ t: f_{qc}(t) =0 \} \big)^{-1}$ is equal to the tracial rank \cite{PSSTW} of $G$.
\end{enumerate}
\end{prop}

It is well known that the fractional chromatic number gives a lower bound on the {\it chromatic number}, $\chi(G)$, of the graph.  The two other ranks were introduced to give lower bounds on two quantum versions of the chromatic numbers. Man\v{c}inska and Roberson \cite{Ro} proved that their projective rank is a lower bound on the standard quantum chromatic number of a graph, $\chi_q(G)$. In \cite{PT} several variations of the standard quantum chromatic number were introduced, including the commuting quantum chromatic number $\chi_{qc}(G)$, and in \cite{PSSTW} it was shown that the tracial rank of a graph is a lower bound on $\chi_{qc}(G)$.

Thus, in a certain sense, these functions measure how small one can keep this total correlation of the traces once one has gone beyond the point where it can be 0. 

There are many other reasons for studying these functions.  We will show later that if Connes' embedding conjecture has an affirmative answer then necessarily, $f_{q}(t)=f_{qc}(t)$, for all $0 \le t \le 1$ and for every graph. Thus, attempting to compute these functions may give us some insight into this conjecture.  These graphs are also related to Tsirelson's conjectures about various models for quantum probability densities. 

We will prove that if the set of quantum correlations is closed for $|G|=n$ inputs and 2 outputs, then necessarily the function $f_q(t)$ is ``piecewise'' linear for vertex and edge transitive graphs.  The core of our proof that $C_q(5,2)$ is not closed is then to show that for the complete graph on five vertices, the function $f_q(t)$ is not piecewise linear.

\section{Preliminaries}
Recall that a {\it positive operator valued measure} ({\it POVM}) is a set $\{R_i\}_{i=1}^k$ of positive operators on some Hilbert space $\cl H$ with $\sum_{i=1}^k R_i=I$. Also a {\it projection valued measure} ({\it PVM}) is a set $\{P_i\}_{i=1}^k$ of projections on some Hilbert space $\cl H$ with $\sum_{i=1}^k P_i=I$. Clearly every PVM is a POVM.

\begin{defn}\label{loc-def}
The set $C_{loc}(n,k)$ is the closed convex hull of all product distributions $\big( p(i,j|v,w) \big), 1 \le v,w \le n, 1 \le i,j \le k$ given by \begin{align*}
p(i,j|v,w)=p^1(i|v)p^2(j|w),
\end{align*} where for $\ell=1,2$, $p^\ell(i|v)\geq 0$ satisfy $\sum_{i=1}^kp^\ell(i|v)=1$, namely, form
a set of $k$-outcome probability distributions indexed by $1\leq w\leq n$. Elements of $C_{loc}(n,k)$ are called \textit{classical correlations}.
\end{defn}

\begin{defn}\label{q-def}
An $n^2k^2$-tuple, $\big( p(i,j|v,w) \big), 1 \le v,w \le n, 1 \le i,j \le k$, is called a {\it quantum correlation} if there exist PVMs $\{P_{v,i}\}_{i=1}^k$ and $\{Q_{w,j}\}_{j=1}^k$ in finite dimensional Hilbert spaces $\cl H_A$ and $\cl H_B$, respectively, together with a unit vector $h\in \cl H_A\otimes \cl H_B$ such that \begin{align*}
p(i,j|v,w) = \langle (P_{v,i}\otimes Q_{w,j})h,h \rangle.
\end{align*} The set of all such tuples $\big( p(i,j|v,w) \big)$ arising from all choices of finite dimensional Hilbert spaces $\cl H_A, \cl H_B$, all PVMs and all unit vectors $h$ is called the set of quantum correlations and is denoted by $C_q(n,k)$.
\end{defn}

If we relax Definition \ref{q-def} by removing the restriction of finite dimensionality on the Hilbert spaces $\cl H_A$ and $\cl H_B$, but keeping everything the same, we get a larger set of correlations called the set of \textit{spatial quantum correlations}, denoted by $C_{qs}(n,k)$.

\begin{defn}\label{qc-def}
An $n^2k^2$-tuple, $\big( p(i,j|v,w) \big), 1 \le v,w \le n, 1 \le i,j \le k$, is called a {\it commuting quantum correlation} if there exist PVMs $\{P_{v,i}\}_{i=1}^k$ and $\{Q_{w,j}\}_{j=1}^k$ in a single (possibly infinite dimensional) Hilbert space $\cl H$ satisfying $P_{v,i}Q_{w,j}=Q_{w,j}P_{v,i}$ (hence the name commuting) together with a unit vector $h\in \cl H$ such that \begin{align*}
p(i,j|v,w) = \langle (P_{v,i}Q_{w,j})h,h \rangle.
\end{align*} The set of all such tuples $\big( p(i,j|v,w) \big)$ arising from all choices of Hilbert space $\cl H$, all PVMs and all unit vectors $h$ is called the set of commuting quantum correlations denoted by $C_{qc}(n,k)$.
\end{defn}

\begin{remark}
If we replace PVMs by POVMs in Definitions \ref{q-def} and \ref{qc-def} we still get the same correlation sets. For the $r=q$ case the equivalence can be shown using a Naimark dilation argument, while the $r=qc$ case is more difficult and a proof can be found in \cite{PT}. This can also be found in Proposition
3.4 of \cite{Fritz}, and also as Remark 10 of \cite{JNPPSW}.
\end{remark}

\begin{remark}\label{remark-q-in-qc}
We have that $C_{loc}(n,k)\subseteq C_q(n,k)\subseteq C_{qc}(n,k)$ for all $n,k\in \bb N$ and these are characterized as follows. By Theorem 5.3 in \cite{PSSTW}, an $n^2k^2$-tuple $(p(i,j|v,w))$ belongs to $C_q(n,k)$ if and only if $(p(i,j|v,w))\in C_{qc}(n,k)$ and has a realization as described in Definition~\ref{qc-def} where the Hilbert space $\mathcal{H}$ in its realization is finite dimensional.
Similarly, by Remark 5.4 in \cite{PSSTW}, a tuple $(p(i,j|v,w))$ belongs to $C_{loc}(n,k)$ if and only if $(p(i,j|v,w))\in C_{qc}(n,k)$ and all the operators in its realization commute.
\end{remark}

There are two other sets of probabilistic correlations that we wish to consider. 

\begin{defn}\label{vec-corr-def}
We call an $n^2k^2$-tuple, $\big( p(i,j|v,w) \big), 1 \le v,w \le n, 1 \le i,j \le k$, a {\it vectorial correlation} provided that there is a Hilbert space $\cl H$ and vectors $x_{v,i}, y_{w,j}, h \in \cl H$, such that:
\begin{itemize}
\item $\|h\|=1$,
\item $\langle x_{v,i}, x_{v,j} \rangle =0, \forall v, \forall i \ne j$,
\item $\langle y_{w,i}, y_{w,j} \rangle =0, \forall w, \forall i \ne j$,
\item $h = \sum_i x_{v,i} = \sum_j y_{w,j}, \, \forall v,w$,
\item $p(i,j|v,w) = \langle x_{v,i}, y_{w,j} \rangle \ge 0, \, \forall v,w,i,j$.
\end{itemize}
We denote the set of all vectorial correlations by $C_{vect}(n,k)$.
\end{defn} 

Since all of the inner products appearing in the above definition are real, there is no generality lost in requiring $\cl H$ to be a real Hilbert space as well.

These correlations have been studied at other places in the literature, see for example \cite{NGHA} where they are referred to as {\it almost quantum correlations} and they are also essentially the first level of the NPA hierarchy \cite{NPA}.

\begin{defn}\label{def:ns}
We call an $n^2k^2$-tuple,  $\big( p(i,j|v,w) \big), 1 \le v,w \le n, 1 \le i,j \le k$, a {\it nonsignalling correlation} provided that:
\begin{itemize}
\item $p(i,j|v,w) \ge 0, \, \forall v,w,i,j$,
\item $\sum_{i,j} p(i,j|v,w) =1, \forall v,w$,
\item $\sum_j p(i,j|v, w) = \sum_j p(i,j|v,w^{\prime}), \, \forall i,v,w,w^{\prime},$   
\item $\sum_i p(i,j|v,w) = \sum_i p(i,j|v^{\prime}, w), \, \forall j,v,v^{\prime},w.$
\end{itemize}
We let $C_{ns}(n,k)$ denote the set of all such nonsignalling correlations. Finally, given a nonsignalling correlation, we set \begin{align*}
p_A(i|v) = \sum_j p(i,j|v,w), \qquad p_B(j|w) = \sum_i p(i,j|v,w),
\end{align*} and refer to these as the {\it marginal densities}. Note that these marginal densities make sense because of the last two properties of a nonsignalling correlation.
\end{defn}

All the correlation sets defined above are related in the following way: \begin{align}
C_{loc}(n,k)\subseteq C_q(n,k) \subseteq C_{qs}(n,k) \subseteq C_{qc}(n,k) \subseteq C_{vect}(n,k) \subseteq C_{ns}(n,k) \subseteq \bb R^{n^2k^2},
\end{align} for all $n,k\in \bb N$, and they are all convex sets \cites{Ts1, Fritz}. Notice that nonsignalling correlations are the largest set of tuples that behave like conditional probability densities and have well-defined marginal densities.

It is known \cites{Ts1, Fritz} that the sets $C_{loc}(n,k), C_{qc}(n,k), C_{vect}(n,k)$ and $C_{ns}(n,k)$ are all closed sets in $\bb R^{n^2k^2}$ for all $n,k\in \bb N$, while $C_q(n,k)$ and $C_{qs}(n,k)$ are not closed for some large values of $n,k$ as shown by Slofstra in \cite{Sl17}. Set $C_{qa}(n,k)=\overline{C_q(n,k)}$. Thus, we have
\begin{align*}
C_q(n,k)\subseteq C_{qs}(n,k) \subseteq C_{qa}(n,k) \subseteq C_{qc}(n,k).
\end{align*} Note that, from the work of Slofstra in \cite{Sl17}, $C_q(n,k)$ and $C_{qs}(n,k)$ are proper subsets of $C_{qa}(n,k)$ for some value of $n$ and $k$. Whether or not they are different for all values of $n,k$ is unknown. From the work in \cite{JNPPSW} and \cite{Oz13} it is known that $C_{qa}(n,k)=C_{qc}(n,k)$ for all $n,k\in \bb N$ is equivalent to Connes' embedding conjecture.

\begin{remark}\label{rem:autom}
For each permutation $\pi$ of $\{1,\ldots,n\}$, we have the corresponding affine self-map $\beta_\pi$ of $C_{ns}(n,k)$ given by \begin{align*}
\beta_\pi:\big(p(i,j|v,w)\big)\mapsto\big(p(i,j|\pi^{-1}(v),\pi^{-1}(w))\big).
\end{align*} These form an action of the group $S_n$ on $C_{ns}(n,k)$. By restriction, they induce actions on $C_r(n,k)$ for $r\in\{loc,q,qa,qc,vect\}$. To see that these restrictions are indeed self-maps, for $r=ns$, given systems $(x_{v,i})$ and $(y_{w,j})$ of vectors that realize a given vectorial correlation $p=(p(i,j|v,w))$, the systems $(x_{\pi^{-1}(v),i})$ and $(y_{\pi^{-1}(w),j})$ of vectors realize $\beta_\pi(p)$. Similarly, for $r\in\{loc,q,qs,qc\}$, applying permutations to systems of projections that realize a given $p\in C_r(n,k)$ show $\beta_\pi(p)\in C_r(n,k)$. The case of $r=qa$ now follows by taking closures.
\end{remark}

\begin{remark}\label{rem:refl}
Exactly analogously to the in the previous remark, we get an action $\sigma\mapsto\gamma_\sigma$ of $S_k$ on $C_r(n,k)$ for each $r\in\{loc,q,qs,qa,qc,vect,ns\}$ given by \begin{align*}
\gamma_\sigma:\big(p(i,j|v,w)\big)\mapsto\big(p(\sigma^{-1}(i),\sigma^{-1}(j)|v,w)\big).
\end{align*} We will use this only in the case $k=2$, when for the order-two transposition $\sigma:0\leftrightarrow1$, we get the {\em reflection} $R=\gamma_\sigma$.
\end{remark}

A correlation $\big( p(i,j|v,w) \big)$ is called {\it synchronous} if $p(i,j|v,v)=0$ for all $1\leq v\leq n$ and for all $i\neq j$. For $r\in \{loc,q,qa,qs,qc,vect,ns\}$, we let $C_r^s(n,k)$ denote the subset of all synchronous correlations. These synchronous subsets are themselves convex and they satisfy \begin{align}\label{syn-hier}
C^s_{loc}(n,k) \subseteq C^s_q(n,k) \subseteq C_{qs}^s(n,k) \subseteq C^s_{qa}(n,k) \subseteq C^s_{qc}(n,k) \subseteq C^s_{vect}(n,k) \subseteq C^s_{ns}(n,k),
\end{align} with all of the containments known to be proper for some values of $n,k$ (see, for example, \cites{PT, PSSTW}), except for the case $C^s_{qa}(n,k) \subseteq C^s_{qc}(n,k)$, where equality for all values of $n$ and $k$ is known, by~\cite{DP16} (see also~\cite{KPS}),
to be equivalent to Connes' embedding conjecture. All of these synchronous subsets are known (see~\cite{PT}) to be closed sets for all $n,k\in \bb N$, except $C_q^s(n,k)$ and $C_{qs}^s(n,k)$.

\begin{remark}\label{rem:R}
The action $\beta$ and the reflection $R$ from Remarks~\ref{rem:autom} and~\ref{rem:refl}
restrict to actions on $C_r^s(n,k)$, for each $r\in\{loc,q,qs,qa,qc,vect,ns\}$.
\end{remark}

We now provide some characterization of these synchronous subsets.

\begin{thm}[Theorem 5.5, \cite{PSSTW}]\label{qcsyn}
Let $(p(i,j|v,w))\in C_{qc}^s(n,k)$ be realized with PVMs $\{P_{v,i}\}_{i=1}^k$ and $\{Q_{w,j}\}_{j=1}^k$ in some $B(\cl H)$ satisfying $P_{v,i}Q_{w,j}=Q_{w,j}P_{v,i}$ and with some unit vector $h\in\mathcal{H}$ such that $p(i,j|v,w)=\left\langle P_{v,i}Q_{w,j} h,h \right\rangle$. 
Then \begin{enumerate}
\item $P_{v,i}h=Q_{v,i}h, \; \forall v,i$,
\item $p(i,j|v,w)=\langle (P_{v,i}P_{w,j})h,h\rangle=\langle (Q_{w,j}Q_{v,i})h,h\rangle=p(j,i|w,v)$,
\item Let $\mathcal{A}$ be the $C^*$-algebra in $B(\cl H)$ generated by the family $\{P_{v,i}\}_{v,i}$ and define $\tau:\mathcal{A}\rightarrow\bb C$ by $\tau(X)=\left\langle Xh,h\right\rangle$. Then $\tau$ is a tracial state on $\mathcal{A}$ and $p(i,j|v,w)=\tau(P_{v,i}P_{w,j})$.
\end{enumerate} Conversely, let $\mathcal{A}$ be a unital $C^*$-algebra equipped with a tracial state $\tau$ and with $\{e_{v,i}\}_{v,i}\subset \mathcal{A}$ a family of projections such that $\sum_{i=1}^ke_{v,i}=1$ for all $v$. Then $(p(i,j|v,w))$ defined by $p(i,j|v,w)=\tau(e_{v,i}e_{w,j})$ is an element of $C_{qc}^s(n,m)$. That is, there exists a Hilbert space $\mathcal{H}$, a unit vector $h\in \mathcal{H}$ and mutually commuting PVMs $\{P_{v,i}\}_i$ and $\{Q_{w,j}\}_j$ on $\mathcal{H}$ such that \begin{align*}
p(i,j|v,w)=\langle (P_{v,i}Q_{w,j})h,h\rangle=\langle (P_{v,i}P_{w,j})h,h\rangle=\langle (Q_{w,j}Q_{v,i})h,h\rangle
\end{align*}
\end{thm}

This theorem and Remark \ref{remark-q-in-qc} lead to the following characterization of $C_{loc}^s(n,k)$ and $C_q^s(n,k)$.

\begin{cor}[Corollary 5.6, \cite{PSSTW}]\label{cor-q-loc}
We have that $(p(i,j|v,w)) \in C_q^s(n,k)$ (respectively, $C_{loc}(n,k)$) if and only if there exists a finite dimensional (respectively, abelian) $C^*$-algebra $\mathcal{A}$ with a tracial state $\tau$ and with a generating family $\{e_{v,i}:1\leq v\leq n, 1\leq i\leq k\}\subseteq \mathcal{A}$ of projections such that $\sum_{i=1}^ke_{v,i}=1$ for all $v$ and $p(i,j|v,w)=\tau(e_{v,i}e_{w,j})$ for all $i,j,v,w$.
\end{cor}

\begin{remark}\label{syn-vect-corr}
In \cite{PT}, it is shown that if the collection of vectors $x_{v,i}, y_{w,j}, h\in \cl H$ as in Definition \ref{vec-corr-def} define a synchronous vectorial correlation, then necessarily, $x_{v,i} =y_{v,i}$, for all $v, i$.
\end{remark}

\begin{thm}[Theorem 3.6, \cite{KPS}]
We have that $\overline{C_q^s(n,k)} = C_{qa}^s(n,k)$ for all $n,k\in\bb N$.
\end{thm}

\section{Basic Properties of the Graph Correlation Function}\label{sec:basicprops}
In this section we define the graph correlation functions $f_r$ and we prove some basic facts about their behaviour.

For each $t\in[0,1]$, we consider the slice
\[%\begin{equation}\label{eq:Gammat}
\Gamma_r(t)=\{(p(i,j|v,w))\in C_r^s(n,2) : p_A(0|v) = p_B(0|w) =t, \, \forall v,w\}
\]%\end{equation}
of $C_r^s(n,2)$, where $p_A$ and $p_B$ are the marginals from Definition~\ref{def:ns}.
We observe that each $\Gamma_r(t)$ is nonempty and convex.

Given a graph $G=(V,E)$ on $n$ vertices, we consider the affine function $F$ on $C_{ns}(n,2)$ given by
\[
F:(p(i,j|v,w))\mapsto\sum_{(v,w) \in E} p(0,0|v,w)
\]
For each $r \in \{ loc, q, qa, qc, vect, ns\}$ and $t\in[0,1]$, we let
\begin{equation}\label{f-def}
f_r(t)=\inf\{F(p):p\in\Gamma_r(t)\}.
\end{equation}

By Theorem \ref{qcsyn} and Corollary \ref{cor-q-loc} this new definition of $f_r(t)$ agrees with the one defined using Equation~\eqref{f-qc-def} and its variants, when $r\in\{loc,q,qc\}$. Moreover, the inclusions \eqref{syn-hier} imply that
\begin{align}\label{f-hier}
f_{loc}(t) \ge f_q(t)  =  f_{qa}(t) \ge f_{qc}(t) \ge f_{vect}(t) \ge f_{ns}(t)\geq 0.
\end{align} Notice that for $r\in \{loc, qa, qc, vect, ns\}$, the set $C_r^s(n,k)$ is closed and thus the infimum in \eqref{f-def} is attained for all $0\leq t\leq 1$.

\begin{prop}\label{f-ns}
If $G=(V,E)$ is a graph on $n$ vertices, then \begin{align*}
f_{ns}(t) = \begin{cases} 
0 &\text{ if } 0\leq t\leq \frac{1}{2}, \\
|E|(2t-1) &\text{ if } \frac{1}{2} \leq t\leq 1.
\end{cases}
\end{align*}
\end{prop}

\begin{proof}
Since $C_{ns}^s(n,2)$ is a closed set, given $t\in [0,1]$ there exists a correlation $(p(i,j|v,w))\in C_{ns}^s(n,2)$ such that $p_A(0|v)=p_B(0|w)=t$ for all $v,w\in V$ and $f_{ns}(t)=\sum_{(v,w)\in E}p(0,0|v,w)$. Since $p_A(0|v)=t$, $\sum_{i,j}p(i,j|v,v)=1$, and using the fact that the correlation is synchronous, we have \begin{align}\label{eq:rel1}
p(0,0|v,v) = t, \qquad p(0,1|v,v)=p(1,0|v,v)=0, \qquad p(1,1|v,v)=1-t.
\end{align} If $(v,w)\in E$, then using the nonsignalling conditions with Equations \eqref{eq:rel1} we get the equations \begin{align*}
p(0,0|v,w)+p(0,1|v,w) &= p(0,0|v,w)+p(1,0|v,w) = t, \\
p(0,1|v,w)+p(1,1|v,w) &= p(1,0|v,w)+p(1,1|v,w) = 1-t,
\end{align*} which have the solution,
\begin{equation}\label{eq:psoln}
\begin{gathered}
p(0,1|v,w) = p(1,0|v,w) = t- p(0,0|v,w), \\
p(1,1|v,w) = 1-2t+p(0,0|v,w).
\end{gathered}
\end{equation} Since these are probabilities we must also have \begin{align*}
p(0,0|v,w) \geq 0, \qquad t- p(0,0|v,w)\geq 0, \qquad 1-2t+p(0,0|v,w)\geq 0,
\end{align*} which yields
\begin{equation}\label{eq:pineq}
\max\{0,2t-1\}\leq p(0,0|v,w)\leq t.
\end{equation} Furthermore, choosing any values for $p(0,0|v,w)$ such that~\eqref{eq:pineq} and~\eqref{eq:rel1} are satisfied and then assigning the other values of $p(i,j|v,w)$
using~\eqref{eq:psoln}, we do get an element of $C_{ns}^s(n,2)$.
This shows that the choice 
\begin{align*}
p(0,0|v,w) = \max\{0,2t-1\} = \begin{cases} 0 &\text{ if } 0\leq t\leq \frac{1}{2} \\ 2t-1 &\text{ if } \frac{1}{2}\leq t\leq 1\end{cases},
\end{align*} yields an element of $C_{ns}^s(n,2)$, whereby the desired value of $f_{ns}(t)$ is attained.
\end{proof}

The following proposition tells that it suffices to describe the functions $f_r$ on the interval $\left[0,\frac{1}{2}\right]$.

\begin{prop}\label{convex}
Let $G=(V,E)$ be a graph on $n$ vertices.
Then $f_{r}$ is a convex function for all $r\in \{loc, q, qa, qc, vect, ns\}$, and \begin{align}\label{symm-about-half}
f_r(1-t)=|E|(1-2t)+f_r(t), \qquad t\in [0,1].
\end{align}
\end{prop}

\begin{proof}
By the convexity of $C_r^s(n,2)$, for each $t_1,t_2,\lambda\in[0,1]$, we have
\[
\lambda\Gamma_r(t_1)+(1-\lambda)\Gamma_r(t_2)\subseteq\Gamma_r(\lambda t_1+(1-\lambda)t_2).
\]
Applying $F$, we have
\begin{multline*}
\lambda F(\Gamma_r(t_1))+(1-\lambda)F(\Gamma_r(t_2))=F(\lambda\Gamma_r(t_1)+(1-\lambda)\Gamma_r(t_2)) \\
\subseteq F(\Gamma_r(\lambda t_1+(1-\lambda)t_2)).
\end{multline*}
Taking infima implies
\[
\lambda f_r(t_1)+(1-\lambda)f_r(t_2)
=\inf(F(\lambda\Gamma_r(t_1)+(1-\lambda)\Gamma_r(t_2)))
\ge f_r(\lambda t_1+(1-\lambda)t_2),
\]
namely, that $f_r$ is convex.

To prove~\eqref{symm-about-half}, we use the reflection map $R:C_r^s(n,2)\to C_r^s(n,2)$ described in Remarks~\ref{rem:refl} and~\ref{rem:R}.
Using~\eqref{eq:rel1} we see that $R$
maps $\Gamma_r^s(t)$ onto $\Gamma_r^s(1-t)$ and using~\eqref{eq:psoln} we see $F\circ R(p)=|E|(1-2t)+F(p)$ for every $p\in C_r^s(n,2)$. \end{proof}

Recall that given a graph, $G=(V,E)$, a {\it graph automorphism} is a bijective function $\pi: V \to V$ such that $(v,w) \in E$ if and only if $(\pi(v), \pi(w)) \in E$. We let $\text{Aut}(G)$ denote the group of all graph automorphisms of $G$. A graph is called {\it vertex transitive} if for every $v,w \in V$ there is a graph automorphism $\pi$ with $\pi(v)=w$. A graph is called {\it edge transitive} if for every $(v,w),(x,y)\in E$, there is a graph automorphism $\pi$ with $(\pi(v),\pi(w))=(x,y)$.

\begin{prop}\label{f-vertex-edge-transitive}
If $G=(V,E)$ is a vertex and edge transitive graph on $n$ vertices, then for every $r \in \{ loc, q, qa, qc, vect, ns\}$
and every $t\in[0,1]$, we have \begin{align*}
f_r(t)=\inf\{F(p):p\in\Gammat_r(t)\},
\end{align*} where \begin{align*}
\Gammat_r(t)=\big\{p=(p(i,j|v,w))\in\Gamma_r(t): p(0,0|v,w)=p(0,0|x,y),\; \forall\,(v,w),\,(x,y)\in E\big\}.
\end{align*}
\end{prop}

\begin{proof}
Using the convexity of $\Gamma_r(t)$ and the vertex and edge transitivity of the graph $G$, it is not hard to show that the map \begin{align*}
p\mapsto\frac1{|\text{Aut}(G)|}\sum_{\pi\in\text{Aut}(G)}\beta_\pi(p),
\end{align*} where the set map $\beta_\pi$ is described in Remark~\ref{rem:autom}, maps $\Gamma_r(t)$ into $\widetilde{\Gamma}_r(t)$.
Since each
$\beta_{\pi}$ leaves the function $F$ invariant, it follows that the above map also leaves $F$ invariant. But then by Equation \eqref{f-def} we get $f_r(t)=\inf\{F(p):p\in\Gammat_r(t)\}$.
\end{proof}

\begin{remark}
Combining Proposition \ref{f-vertex-edge-transitive} with our characterizations of synchronous commuting quantum correlations (Theorem \ref{qcsyn}) in terms of traces,
we see that for a vertex and edge transitive graph $G=(V,E)$, we have that $f_{qc}(t) = s$ where $s$ is the smallest value for which there exists a $C^*$-algebra $\cl A$ with a tracial state $\tau$ and projections $P_v \in \cl A$ such that  $\tau(P_v) =t, \, \forall v \in V$ and $\tau(P_vP_w) =\frac{s}{|E|}, \, \forall (v,w) \in E$.
\end{remark}

\begin{remark}\label{rem:psolve}
Let $r\in \{loc, q, qa, qc, vect, ns\}$ and let $(p(i,j|v,w))\in C_r^s(n,2)$ be such that $p_A(0|v)=p_B(0|w)=t$
for all $v,w\in V$ and $p(0,0|v,w)=\frac{s}{|E|}$ for all $(v,w)\in E$.
The synchronous condition implies $t=p_A(0|v)=p(0,0|v,v)+p(0,1|v,v)=p(0,0|v,v)$, so that \begin{align}\label{eq:rel2}
p(0,0|v,v)=t, \qquad p(0,1|v,v)=p(1,0|v,v) = 0, \qquad p(1,1|v,v) = 1-t.
\end{align} If $(v,w)\in E$, then using the nonsignalling conditions of Definition~\ref{def:ns} with Equations \eqref{eq:rel2} we must have
\begin{gather*}
p(0,0|v,w) = \frac{s}{|E|}, \qquad p(0,1|v,w) = p(1,0|v,w) = t-\frac{s}{|E|}, \\
p(1,1|v,w) = 1-2t+\frac{s}{|E|}.
\end{gather*}
Since these are probabilities, we must have \begin{align}\label{eq:constraint-on-s}
0 \leq \max\{0,2t-1\} \leq \frac{s}{|E|} \leq t.
\end{align}
\end{remark}

\begin{prop}\label{piecewise}
Let $G=(V,E)$ be a vertex and edge transitive graph on $n$ vertices and let $t\in [0,1]$ be irrational.
Suppose that the value of $f_q(t)$ is attained in the infimum~\eqref{f-def} defining it. Then there is a nondegenerate interval $[r,s]$ having rational endpoints such that $t\in[r,s]$ and the restriction of $f_q$ to $[r,s]$ is linear.
\end{prop}

\begin{proof}
Since the value $f_q(t)$ is attained, there is a finite dimensional $C^*$-algebra $\cl A$ generated by projections
$\{P_v: v \in V\}$ and equipped with a trace $\tau: \cl A \to \bb C$ with $\tau(P_v) =t$ for all $v\in V$ and such that \begin{align*}
f_q(t) = \sum_{(v,w) \in E} \tau(P_vP_w).
\end{align*} Since $\cl A$ is finite dimensional, we may write $\cl A= \bigoplus_{l=1}^m \bb M_{n_l}$ and $\tau = \oplus_{l=1}^m \lambda_{l} \text{tr}_{n_l}$, where $\lambda_{l} >0$ with $\sum_{l=1}^m \lambda_l =1$, and where $\text{tr}_{n_l}: \bb M_{n_l} \to \bb C$ denotes the normalized trace on matrices, i.e., $\text{tr}_{n_l}([x_{i,j}]) = \frac{1}{n_l} \sum_{i=1}^{n_l} x_{i,i}$; moreover, we have
$P_v = \oplus_{l=1}^m P_{v,l}$  for projections $P_{v,l}\in \bb M_{n_l}$.
Let $\text{Aut}(G)$ denote the group of graph automorphisms of the graph $G$ and set $N= |\text{Aut}(G)|$. For $v\in V$ and $1\leq l\leq m$, set \begin{align*}
\widetilde{P}_{v,l} = \oplus_{\pi \in \text{Aut}(G)} P_{\pi(v), l} \in \bigoplus_{\pi \in \text{Aut}(G)} \bb M_{n_l}=: \cl A_l.
\end{align*}  Define a trace,  $\tau_l: \cl A_l \to \bb C$, by \begin{align*}
\tau_l\left( \oplus_{\pi\in \text{Aut}(G)} X_{\pi} \right) = \frac{1}{N} \sum_{\pi\in \text{Aut}(G)} \text{tr}_{n_l}(X_{\pi}).
\end{align*} Given any $v,w \in V$ if we fix $\rho \in \text{Aut}(G)$ such that $\rho(v)=w$, then \begin{align*}
\tau_l(\widetilde{P}_{w,l}) = \frac{1}{N} \sum_{\pi \in \text{Aut}(G)} \text{tr}_{n_l}(P_{\pi(w),l}) = \frac{1}{N}\sum_{\pi \in \text{Aut}(G)} \text{tr}_{n_l}(P_{\pi \rho(v), l}) = \tau_l(\widetilde{P}_{v,l}),
\end{align*} which is some fixed rational number $r_l$. After a permutation we may assume that these rational numbers $r_l$ are arranged in non-decreasing order.

Thus,  $\{ \widetilde{P}_{v,l}: v \in V \}$ is a feasible set for the definition of $f_q(r_l)$ and hence we have that \begin{align*}
f_q(r_l) \le \sum_{(v,w) \in E} \tau_l( \widetilde{P}_{v,l} \widetilde{P}_{w,l}).
\end{align*}

Now, we set $\widetilde{\cl A} = \oplus_{l=1}^m \cl A_l$, and define a normalized trace $\widetilde{\tau}: \widetilde{\cl A} \to \bb C$ by $\widetilde{\tau}(\oplus_{l=1}^m Y_l) = \sum_{l=1}^m \lambda_l \tau_l(Y_l)$. Define projections $\widetilde{P}_v$ in $\widetilde{\cl A}$ by $\widetilde{P}_v = \oplus_{l=1}^m \widetilde{P}_{v,l}$. Then we have that \begin{align*}
\widetilde{\tau}(\widetilde{P}_v) = \sum_{l=1}^m \lambda_l \tau_l(\widetilde{P}_{v,l}) = \frac{1}{N} \sum_{l=1}^m \sum_{\pi \in \text{Aut}(G)} \lambda_l \text{tr}_{n_l}(P_{\pi(v),l}) = 
\frac{1}{N} \sum_{\pi \in \text{Aut}(G)}   \tau(P_{\pi(v)}) = t,
\end{align*} while a similar calculation shows that $\sum_{(v,w) \in E} \widetilde{\tau}( \widetilde{P}_v \widetilde{P}_w) = f_q(t)$. Thus, \begin{align*}
f_q(t) = \sum_{(v,w)\in E}\sum_{l=1}^m \lambda_l \tau_l(\widetilde{P}_{v,l} \widetilde{P}_{w,l}) \ge \sum_{l=1}^m \lambda_l f_q(r_l).
\end{align*} By Proposition~\ref{convex}, $f_q$ is a convex function and so we have \begin{align*}
f_q(t) = \sum_{l=1}^m \lambda_l f_q(r_l),
\end{align*} and so we must have that $f_q(r_l) = \sum_{(v,w) \in E} \tau_l(\widetilde{P}_{v,l}\widetilde{P}_{w,l})$.

But this is exactly the equality case of Jensen's inequality, which holds if and only if either all the points in the convex combination are the same or the function is piecewise linear on an interval containing the points. Since $t$ is irrational, the points $r_l$ cannot all be same and this forces the function $f_q$ to be linear on an interval containing the points $r_l$.
\end{proof}

The following is straightforward to prove. See, for example, Proposition~5.2 of~\cite{DPP}.

\begin{lemma}\label{commuting-lemma}
Let $\mathcal{A}$ be a unital $C^*$-algebra with a faithful tracial state $\tau$. Let $A$ and $P$ be hermitian elements in $\mathcal A$. If $AP-PA\neq 0$, then there exists $H=H^*\in \mathcal{A}$ such that, letting $f(t)=\tau(A(e^{iHt}Pe^{-iHt}))$ for $t\in \bb R$, we have $f'(0)> 0$.
\end{lemma}

The following result is not used in the proofs of other results in this paper (however, see the Appendix, where a similar argument is used).
But it was in a sense the key result for our proof of Theorem~\ref{thm:notclosed}, because it led us to ask about
scalar multiples of the identity realized as sums of projections, and to find the results~\cite{KRS} of 
Kruglyak, Rabanovich, and Samo\u\i lenko.

\begin{prop}\label{localcommuting}
Let $G=(V,E)$ be a graph on $n$ vertices, and assume that $\tau: \cstar \to \bb C$ is a tracial state
(respectively, finite dimensional tracial state) such that $\tau(e_v) = t$ for all $v\in V$ and 
$f_{qc}(t)$
(respectively, $f_q(t)$) is equal to
$\sum_{(v,w) \in E} \tau(e_ve_w)$.
 Set $p_v = \sum_{\{w\,:\,(v,w) \in E\}} e_w$. If $\pi:\cstar \to B(\cl H)$ is the GNS representation of $\tau$, then $\pi(e_v)\pi(p_v) = \pi(p_v) \pi(e_v)$.
\end{prop}

\begin{proof}
Fix $v\in V$. Let $\pi:\cstar\to B(\cl H)$ be the GNS representation of $\tau$ with $\tau(a)=\langle \pi(a)\psi,\psi \rangle$ for all $a\in \cstar$ and for some cyclic vector $\psi\in \cl H$. Let $\cl B=\pi(\cstar)\subseteq B(\cl H)$ be the image $C^*$-algebra.
Suppose, for contradiction, that $\pi(e_v)$ and $\pi(p_v)$ do not commute. Then by Lemma \ref{commuting-lemma}, there exists $H=H^*\in \cl B$ (therefore $H=\pi(h)$, $h\in \cstar$) such that if 
\[f(t)=\left\langle \pi(e_v)(e^{iHt}\pi(p_v)e^{-iHt})\psi,\psi \right\rangle=\tau(e_v(e^{iht}\pi(p_v)e^{-iht})),\] 
then $f'(0)>0$. Fix some small and negative $t_0$ such that $f(t_0)<f(0)$.

Define for $y\in V$,
\[F_y  = \begin{cases}
\pi(e_v) &\text{ if } y=v, \\
e^{iHt_0}\pi(e_y)e^{-iHt_0} &\text{ if } y\neq v.
\end{cases}\]
 Then each $F_y$ is a projection in $\cl B$ and \begin{align*}
\left\langle F_y\psi,\psi \right\rangle = \left\langle \left(e^{iHt_0}\pi(e_y)e^{-iHt_0}\right)\psi,\psi \right\rangle = \tau\left(e^{iht_0}\pi(e_y)e^{-iht_0}\right) = \tau(e_y)= t.
\end{align*} But for this new set of projections, we have that \begin{align*}
\sum_{(x,y)\in E}&\left\langle F_xF_y\psi,\psi  \right\rangle \\
&=
\begin{aligned}[t]
\sum_{\{w\,:\,(v,w)\in E\}}\left\langle F_vF_w\psi,\psi \right\rangle
+ \sum_{\{w\,:\,(w,v)\in E\}}\left\langle F_wF_v\psi,\psi \right\rangle& \\
+ \sum_{\{(x,y)\in E\,:\,x\neq v,\,y\neq v\}} \left\langle F_xF_y\psi,\psi  \right\rangle&
\end{aligned} \displaybreak[2]\\
&=
\begin{aligned}[t]
\left\langle F_v\left( \sum_{\{w\,:\,(v,w)\in E\}} F_w\right) \psi,\psi\right\rangle
+\left\langle \left( \sum_{\{w\,:\,(w,v)\in E\}} F_w\right)F_v \psi,\psi\right\rangle& \\
+ \sum_{\{(x,y)\in E\,:\,x\neq v,\,y\neq v\}} \left\langle F_xF_y\psi,\psi  \right\rangle&
\end{aligned} \displaybreak[2]\\
&= 2\,\mathrm{Re}\,\left\langle \pi(e_v) \left( \sum_{\{w\,:\,(v,w)\in E\}}e^{iHt_0}\pi(e_w)e^{-iHt_0}\right) \psi,\psi \right\rangle \\
&\qquad + \sum_{\{(x,y)\in E\,:\,x\neq v,\,y\neq v\}} \left\langle \left(e^{iHt_0}\pi(e_x)e^{-iHt_0} \right) \left(e^{iHt_0}\pi(e_y)e^{-iHt_0} \right) \psi,\psi\right\rangle \displaybreak[2]\\
&= 2\,\mathrm{Re}\,\left\langle \pi(e_v) \left(e^{iHt_0}\pi(p_v)e^{-iHt_0}\right)\psi,\psi \right\rangle + \sum_{\{(x,y)\in E: x\neq v, y\neq v\}} \tau(e^{iht_0}e_xe_ye^{-iht_0}) \displaybreak[2]\\
&=2f(t_0) + \sum_{\{(x,y)\in E\,:\,x\neq v,\,y\neq v\}} \tau(e_xe_y) \\
&< 2f(0) + \sum_{\{(x,y)\in E\,:\,x\neq v,\,y\neq v\}} \tau(e_xe_y) \\ 
&= \tau(e_vp_v)+\tau(p_ve_v)+\sum_{\{(x,y)\in E\,:\,x\neq v,\,y\neq v\}} \tau(e_xe_y) \\
&= \sum_{(x,y)\in E}\tau(e_xe_y) = f_{qc}(t),
\end{align*} where we have used that $(v,w)\in E$ if and only if $(w,v)\in E$. This contradicts the definition of $f_{qc}$. \end{proof}

\begin{thm} \label{posmatrixcond}
Let $G=(V,E)$ be a vertex and edge transitive graph on $n$ vertices and let $t\in [0,1]$.
Then $f_{vect}(t)= s$, where $s$ is the smallest real number satisfying Equation~\eqref{eq:constraint-on-s} and for which there exists an $(n+1)\times (n+1)$ positive semidefinite matrix $P=[p_{i,j}]_{i,j=0}^n$ satsifying
\begin{itemize}
\item $p_{i,j} \ge 0, \forall i,j,$
\item $p_{0,0}=1, p_{i,i}=t, 1 \le i \le n$,
\item $p_{0,j}= p_{j,0} =t, 1 \le j \le n$
\item $p_{i,j} =\frac{s}{|E|}, \forall (i,j) \in E$.
\end{itemize}
\end{thm}
\begin{proof}
Fix $t\in [0,1]$ and let $f_{vect}(t)=s$. Then $s$ must satisfy Equation \eqref{eq:constraint-on-s}.
Since $C_{vect}^s(n,2)$ is closed, by Proposition \ref{f-vertex-edge-transitive} there exists $(p(i,j|v,w))\in C_{vect}^s(n,2)$ such that $p_A(0|v)=p_B(0|w)=t$ for all $v,w\in V$ and $p(0,0|v,w)=\frac{s}{|E|}$ for all $(v,w)\in E$. By Definition \ref{vec-corr-def} and Remark \ref{syn-vect-corr} there exist vectors $\{h,x_{v,0},x_{v,1}\}\subseteq \cl H$ in some Hilbert space $\cl H$ such that \begin{align*}
\|h\|=1, \qquad \langle x_{v,0},x_{v,1} \rangle=0, \qquad h=x_{v,0}+x_{v,1}, \qquad p(i,j|v,w)=\langle x_{v,i},x_{w,j}\rangle.
\end{align*} Set $x_v=x_{v,0}$ and $y_v=x_{v,1}$. Let $x_0=h$ and let $P=[p_{v,w}]_{v,w=0}^n$ be the Grammian of vectors $\{x_0,x_1,\ldots,x_n\}$. Then this matrix is positive semidefinite and satisfies the properties stated in theorem.
For notice that for all $v\in V$ we have
\begin{align*}
\langle x_v,h \rangle &= \langle x_{v,0},x_{v,0}+x_{v,1} \rangle = p(0,0|v,v)+p(0,1|v,v) = p_A(0|v) = t, \\
\|x_v\|^2  &= \langle x_{v,0},x_{v,0} \rangle = \langle x_{v,0},h-x_{v,1} \rangle = \langle x_{v,0},h \rangle = t,
\end{align*} and for all $(v,w)\in E$ we have
\[
\langle x_v,x_w \rangle= \langle x_{v,0},x_{w,0} \rangle = p(0,0|v,w) = \frac{s}{|E|}.
\]

Conversely, given such a matrix $P$ there are vectors $\{x_0,\ldots,x_n\}$ such that $P$ is the Grammian of these vectors. Set $h= x_0$ and $y_v = x_0 - x_v$ for all $1\leq v\leq n$ and observe that $\langle x_v,y_v \rangle = \langle x_v,x_0 - x_v \rangle = p_{0,v} - p_{v,v} = t-t=0$, from which it is easy to construct a synchronous vectorial correlation.
\end{proof} 
  
\begin{prop}
Let $G=(V,E)$ be a graph on $n$ vertices. Then \begin{align*}
f_{q}\left(\frac{1}{2} \right) = f_{qa}\left(\frac{1}{2} \right) = f_{qc}\left(\frac{1}{2} \right) = f_{vect}\left(\frac{1}{2} \right).
\end{align*}
\end{prop}
  
\begin{proof}
From the relations \eqref{f-hier}, it is sufficient to show that $f_q\left(\frac{1}{2} \right)=f_{vect}\left(\frac{1}{2} \right)$.

Let $(p(i,j|v,w))\in C_{vect}^s(n,2)$ be such that $p_A(0|v)=p_B(0|w)=\frac{1}{2}$. By Remark \ref{syn-vect-corr} there exist vectors $\{x_{v,0},x_{v,1},h\}\subset \cl H$ such that $p(i,j|v,w)=\langle x_{v,i},x_{w,j}\rangle$. Without loss of generality we may assume that $\cl H$ is a finite-dimensional real Hilbert space, say of dimension $m$. Set $x_v=x_{v,0}$ for all $v\in V$. Then $\frac{1}{2}=p_A(0|v)=\langle x_v,h \rangle$, and nonsignalling conditions yield, \begin{align*}
p(0,0|v,w) &= p(1,1|v,w) = \langle x_v,x_w \rangle \\
p(0,1|v,w) &= p(1,0|v,w) = \frac{1}{2}- \langle x_v,x_w \rangle.
\end{align*} Define $\widetilde{x}_v = 2x_v-h$ for all $v\in V$. It is easy to verify that each $\widetilde{x}_v$ is a unit vector, and \begin{align*}
p(i,j|v,w) = \frac{1}{4}\left(1+(-1)^{i+j}\langle \widetilde{x}_v,\widetilde{x}_w \rangle \right).
\end{align*}
Recall the representation of the Clifford algebra that is determined by a real linear map
$\cl H\ni x\mapsto C(x)\in\bb M_d$ for some $d$, where each $C(x)$ is self-adjoint and has trace zero and where they satisfy
$C(x)C(y)+C(y)C(x)=2\langle x,y\rangle I_d$.
Thus, when $x$ is a unit vector, $C(x)$ is a symmetry.
We let
\[
P_{v,i} = \frac{I+(-1)^iC(\widetilde x_v)}{2}.
\]
Then each $P_{v,i}$ is a projection and computation shows
\[
\text{tr}_d(P_{v,i}P_{w,j})=\frac{1}{4}\left(1+(-1)^{i+j}\langle \widetilde{x}_v,\widetilde{x}_w \rangle \right) = p(i,j|v,w).
\]
Therefore $(p(i,j|v,w))\in C_q^s(n,2)$ as well and the proposition follows.
\end{proof}

\section{Complete Graphs}
In this section, we compute the function $f_{vect}$ explicitly for the complete graph $K_n$ when $n\geq 3$. We shall then compare the function $f_{vect}$ with the function $f_q$ for $K_5$ to deduce that the set $C_q(5,2)$ is not closed.

\begin{prop}\label{f-vect-complete-graph}
For the complete graph $K_n$ on $n \ge 3$ vertices, we have that \begin{align*}
f_{vect}(t) = \begin{cases}
0, &\text{ if } 0 \le t \le \frac{1}{n}, \\
nt(nt-1), &\text{ if } \frac{1}{n} \le t \le \frac{n-1}{n}, \\
(n^2-n)(2t-1), &\text{ if } \frac{n-1}{n} \le t\le 1.
\end{cases} 
\end{align*}
\end{prop}

\begin{proof}
 We seek the smallest $s$ for which the $(n+1) \times
  (n+1)$ matrix satisfying the conditions of
  Theorem~\ref{posmatrixcond} is positive semidefinite. Applying one
  step of the Cholesky algorithm, this is equivalent to the $n \times
  n$ matrix $Q=[q_{i,j}]$ being positive semidefinite, where $q_{i,i}
  = t-t^2$ and $q_{i,j} = \frac{s}{|E|} -t^2$ for $i \ne j$. Let $J$ be the $n
  \times n$ matrix of all 1's, then \begin{align*}
 Q= \left(t-\frac{s}{|E|}\right)I +\left(\frac{s}{|E|}-t^2\right)J
\end{align*} which has eigenvalues, \begin{align*}
\left\lbrace t-\frac{s}{|E|}, (n-1)\frac{s}{|E|} +t - nt^2 \right\rbrace.
\end{align*} Thus, $Q$ is positive semidefinite if and only if 
\[ \frac{nt^2-t}{n-1} \le \frac{s}{|E|} \le t.\]
Combining this condition with the constraint in \eqref{eq:constraint-on-s} and observing that $\frac{nt^2-t}{n-1} \le t$ for $0 \le t \le 1$, we arrive at \begin{align*}
\max\left\lbrace  0, \frac{nt^2-t}{n-1}\right\rbrace &\leq \frac{s}{|E|}, \text{ when } 0\leq t\leq \frac{1}{2}, \\
\max\left\lbrace 2t-1, \frac{nt^2-t}{n-1} \right\rbrace &\leq \frac{s}{|E|}, \text{ when } \frac{1}{2}\leq t\leq 1.
\end{align*} Simplifying this proves the proposition.
\end{proof}

\begin{thm}\label{thm:notclosed}
The synchronous correlation set $C_q^s(5,2)$ is not closed.
\end{thm}

\begin{proof}
Consider the complete graph $G=K_5$ on five vertices. By Proposition \ref{f-vect-complete-graph} we know that \begin{align*}
f_{vect}(t) = \begin{cases}
0, &\text{ if } 0 \le t \le \frac{1}{5}, \\
5t(5t-1), &\text{ if } \frac{1}{5} \le t \le \frac{4}{5}, \\
20(2t-1), &\text{ if } \frac{4}{5} \le t\le 1.
\end{cases} 
\end{align*} Notice that $f_{vect}(t)$ is quadratic in $t$ on the interval $\left[\frac{1}{5},\frac{4}{5}\right]$. We show that $f_q(t)=f_{vect}(t)=5t(5t-1)$ for all rational $t\in \left[\frac{\sqrt{5}-1}{2\sqrt{5}}, \frac{\sqrt{5}+1}{2\sqrt{5}}  \right]\subset \left[\frac{1}{5},\frac{4}{5}\right]$. This will imply that $f_q$ cannot be linear on any nondegenerate subinterval of $\left[\frac{\sqrt{5}-1}{2\sqrt{5}}, \frac{\sqrt{5}+1}{2\sqrt{5}}  \right]$, so that, by Proposition \ref{piecewise}, it will follow that the value of $f_q(t)$ is not attained for any irrational $t$ in that interval. In this case, $C_q^s(5,2)$ cannot be closed.

From~\eqref{f-hier}, we have $f_q(t)\ge f_{vect}(t)=5t(5t-1)$ when $t\in[\frac15,\frac45]$. Suppose
$t\in \left[\frac{\sqrt{5}-1}{2\sqrt{5}}, \frac{\sqrt{5}+1}{2\sqrt{5}}  \right]$ and $t$ is rational.
We will show $f_q(t)\le 5t(5t-1)$. Since $5t\in \left[\frac{5-\sqrt{5}}{2}, \frac{5+\sqrt{5}}{2}  \right]\cap \bb Q$, by Theorem 6 in \cite{KRS}, it follows that there exist five projections $P_1,\ldots,P_5\in \bb M_k$ for some natural number $k$, such that $P_1+\cdots+P_5=5t \bb I_k$. Define \begin{align*}
\widetilde{P}_i=P_i\oplus P_{i+1}\oplus\cdots\oplus P_{i+4}\in \bb M_k\oplus \bb M_k\oplus\cdots\oplus \bb M_k \subseteq \bb M_{5k}.
\end{align*} Clearly $\sum_{j=1}^5\widetilde{P}_j = 5t\bb I_{5k}$, and also notice that if $\text{tr}_{5k}$ denotes the normalized trace on $\bb M_{5k}$, then \begin{align*}
\text{tr}_{5k}(\widetilde{P}_i)=\frac{1}{5k}\text{Tr}(\widetilde{P}_i)= \frac{1}{5k}\sum_{j=1}^5\text{Tr}(P_j)  = \frac{1}{5k}\text{Tr}\left(\sum_{j=1}^5 P_j \right) =\frac{1}{5k}(5tk)=t.
\end{align*}

Therefore, we have five projections $\widetilde{P}_1,\ldots,\widetilde{P}_5\in \bb M_{5k}$ such that $\text{tr}_{5k}(\widetilde{P}_i)=t$, for all $1\leq i\leq 5$, and $\sum_{j=1}^5\widetilde{P}_j=5t\bb I_{5k}$. Squaring the sum, we get $\sum_{i\neq j}\widetilde{P}_i\widetilde{P}_j=5t(5t-1)\bb I_{5k}$, which, upon taking the normalized trace, yields \begin{align*}
\sum_{i\neq j}\text{tr}_{5k}(\widetilde{P}_i\widetilde{P}_j) = 5t(5t-1).
\end{align*} This implies $f_q(t)=5t(5t-1)$ for all $t\in \left[\frac{\sqrt{5}-1}{2\sqrt{5}}, \frac{\sqrt{5}+1}{2\sqrt{5}}  \right]\cap \bb Q$, completing the proof.
\end{proof}

\begin{remark}
Examining the above proof, we can write down an explicit element of $C_{qa}(5,2)$
that is not an element of $C_q(5,2)$.
Indeed, let $t$ be an irrational element of the interval $\left[\frac{\sqrt{5}-1}{2\sqrt{5}}, \frac{\sqrt{5}+1}{2\sqrt{5}}  \right]$.
Working with the complete graph $K_5$,
since $f_{qa}(t)=f_q(t)=5t(5t-1)$, by Proposition~\ref{f-vertex-edge-transitive} and since $\Gammat_{qa}(t)$ is closed,
there exists 
\[
p=(p(i,j|v,w))\in\Gammat_{qa}(t)\subseteq C_{qa}(5,2)
\]
such that
$p_A(0|v)=p_B(0|w)=t$ for all $v,w\in V$ and $p(0,0|v,w)=\frac t4(5t-1)$ for all $v,w\in V$ with $v\ne w$.
Now using Remark~\ref{rem:psolve}, we calculate:
if $v=w$, then
\[
p(0,0|v,w)=t,\qquad p(0,1|v,w)=p(1,0|v,w)=0,\qquad p(1,1|v,w)=1-t,
\]
while if $v\ne w$, then 
\begin{gather*}
p(0,0|v,w)=\frac14t(5t-1),\qquad p(0,1|v,w)=p(1,0|v,w)=\frac54t(1-t),\\
p(1,1|v,w)=\frac14(1-t)(4-5t).
\end{gather*}
However, since the value $f_q(t)$ is not attained in the infimum defining it, we have $p\notin C_q(5,2)$.
\end{remark}

\begin{cor}
The sets $C_q(5,2)$ and $C_{qs}(5,2)$ are not closed, and $C_{qs}(5,2) \ne C_{qa}(5,2)$.
\end{cor}

\begin{proof}
It is easily seen that if $C_q(5,2)$ were closed then necessarily the subset of synchronous quantum correlations would be closed. Hence, $C_q(5,2)$ is not closed.

Similar reasoning shows that if $C_{qs}(5,2)$ were closed, then $C_{qs}^s(5,2)$ would be closed.
But Theorem~3.10 of~\cite{KPS} shows that $C_{qs}^s(5,2) = C_q^s(5,2)$, and so $C_{qs}(5,2)$ is not closed. 

The last claim follows from the fact that $C_{qa}(5,2)$ is closed.
\end{proof}

%\appendix

\section{Signed games whose synchronous quantum values are not attained}

There is a significant body of research studying the $I_{3322}$ game, which is a 3 input and 2 output game, with the goal of showing that its quantum value is not attained. See \cite{DPP} for references to some of the literature on this game.

We now show how to
turn our examples of non-closure of quantum correlation sets into a collection of signed games, with 5 inputs and 2 outputs, whose synchronous quantum values are not attained.

Consider a game $\cl G$ with $n \ge 5$ inputs $I$ and 2 outputs $\{ 0,1 \}$.  Alice and Bob are rewarded with $+1$
if when they receive the same input they both reply with $0$,
and penalized with $-1$, if when they receive different inputs they respond with $0$. All other cases
have no effect on the game. Assume that the $n$ input pairs $(x,x)$ all are received with the probability $\frac{1-t}{n}$ and that the $n^2-n$ input pairs $(x,y), x \ne y$ are all received with probability $\frac{t}{n^2-n}$ where $0<t<1$.

If $p(i,j|x,y)$ represents the conditional probability density $p(i,j| x,y)$ that Alice replies with $i$ when receiving $x$ and that Bob replies with $j$ when receiving $y$,
where $i,j \in \{0,1\}$ and $1 \le x,y \le n$,
then the expected value is
\[ E(p)= \frac{1-t}{n} \sum_{x=1}^n p(0,0|x,x)  - \frac{t}{n^2-n}  \sum_{x \ne y} p(0,0|x,y).\] Set $A= \frac{1-t}{n}$ and $B= \frac{t}{n^2-n}$.

% = \\ \frac{1-t}{n}\sum_{x=1}^n (1- p(1,0|x,x) - p(0,1|x,x) - p(0,0|x,x)) + \frac{t}{n^2-n} \sum_{x \ne y} p(0,0|x,y) \le \\  1-t - \frac{1-t}{n} \sum_{x=1}^n p(0,0|x,x) + \frac{t}{n^2 -n} \sum_{x \ne y} p(0,0|x,y). \end{multline}

We will now show that for certain values of $t$ the supremum of this expected value over all synchronous quantum strategies is not attained.
By Corollary~\ref{cor-q-loc}, a density arises from a synchronous quantum strategy if and only if it has the form  
\[  p(i,j | x,y) = \tau(E_{x,i}E_{y,j}),\]
where $E_{x,i}$ are projections in a finite dimensional C*-algebra satisfying $E_{x,0} + E_{x,1} =I$ and $\tau$ is a tracial state on that algebra.
Thus, we are trying to compute the supremum of the quantity
\begin{equation}\label{eq:ABtauE}
A \sum_{x=1}^n \tau(E_{x,0}) - B \sum_{x \ne y} \tau(E_{x,0}E_{y,0})
\end{equation}
over all such algebras and traces.
Arguing like in the proof of Proposition~\ref{localcommuting}, one easily shows that if the supremum of this quantity is attained
for a family $(E_{x,i})_{x,i}$ of projections,
then the self-adjoint element $\sum_x E_{x,0}$ must lie in the center of the algebra $\ff{A}$ generated by the family.
If $Q$ is a minimal projection of this center, then the renormalized restriction of $\tau$ to $Q\ff{A} Q$ is a tracial state that together with
the projections $QE_{x,i}$ forms an instance over which we are taking the supremum.
Then the value of~\eqref{eq:ABtauE} is an appropriate convex combination of these instances, so all must yield the same value.
Thus by considering one of these minimal projections, we may without loss of generality assume $\sum_x E_{x,0}=\lambda I$ for some scalar $\lambda$
and that the algebra $\ff{A}$ generated by the collection of $E_{x,0}$ has trivial center, namely, is a matrix algebra $M_p(\bb{C})$ for some integer $p\ge1$.
Since this algebra has a unique tracial state, and this trace takes rational values on all projections, we see that this value of $\lambda$ must be rational.
Moreover, the quantity~\eqref{eq:ABtauE} becomes
\[
A\sum_{x=1}^n\tau(E_{x,0})-B\big(\sum_{1\le x,y\le n}\tau(E_{x,0}E_{y,0})-\sum_{x=y}\tau(E_{x,0}E_{y,0})\big)
=(A+B)\lambda-B\lambda^2.
\]
Since $B>0$, the maximum value of the right-hand expression occurs at the unique value
\begin{equation}\label{eq:bestlam}
\lambda=\lambda^*:=\frac{A+B}{2B}=1-\frac n2+\frac{n-2}{2t}.
\end{equation}
For specificity, let us take $n=5$.
Recall that, by Theorem~6 of~\cite{KRS},
for all rational $\lambda\in \left[\frac{\sqrt{5}-1}{2\sqrt{5}}, \frac{\sqrt{5}+1}{2\sqrt{5}}  \right]$,
there exist
projections $(E_{x,0})_{x=1}^5$ on a finite dimensional Hilbert space such that $\sum_{x=1}^5E_{x,0}=\lambda I$.
Choosing $t$ so that the optimizing value $\lambda^*$ belongs to that interval, we see that the supremum of the quantity~\eqref{eq:ABtauE}
is equal to
\[
(A+B)\lambda^*-B(\lambda^*)^2=\frac{(A+B)^2}{4B}.
\]
However, when $t$ is irrational then $\lambda^*$ is irrational and, as remarked above, this value cannot be realized 
as the quantity~\eqref{eq:ABtauE} for finite dimensional projections $E_{x,0}$ and a trace $\tau$.
Thus, for such values of $t$, the synchronous quantum value of this signed game is not attained.


\begin{thebibliography}{19}

\bib{DP16}{article}{
  author={Dykema, Ken},
  author={Paulsen, Vern},
  title={Synchronous correlation matrices and Connes' embedding conjecture},
  journal={J. Math. Phys.},
  volume={57},
  date={2016},
%  number={1},
  pages={015214, 12},
%  issn={0022-2488},
%  review={\MR{3432742}},
%  doi={10.1063/1.4936751},
}

\bib{DPP}{article}{
  author={Dykema, Ken},
  author={Paulsen, Vern},
  author={Prakash, Jitendra},
  title={The delta game},
  journal={Quantum Inf. Comput.},
  status={to appear},
  eprint={arXiv:1707.06186},
}

\bib{Fritz}{article}{
   author={Fritz, Tobias},
   title={Tsirelson's problem and Kirchberg's conjecture},
   journal={Rev. Math. Phys.},
   volume={24},
   date={2012},
%   number={5},
   pages={1250012, 67},
%   issn={0129-055X},
%   review={\MR{2928100}},
%   doi={10.1142/S0129055X12500122},
}

\bib{JNPPSW}{article}{
   author={Junge, M.},
   author={Navascues, M.},
   author={Palazuelos, C.},
   author={Perez-Garcia, D.},
   author={Scholz, V. B.},
   author={Werner, R. F.},
   title={Connes embedding problem and Tsirelson's problem},
   journal={J. Math. Phys.},
   volume={52},
   date={2011},
%   number={1},
   pages={012102, 12},
%   issn={0022-2488},
%   review={\MR{2790067}},
%   doi={10.1063/1.3514538},
}

\bib{KPS}{article}{
  author={Kim, Se-Jin},
  author={Paulsen, Vern Ival},
  author={Shafhauser, Christopher},
  title={A synchronous game for binary constraint systems},
  eprint={arxiv:1707.01016},
}

\bib{KRS}{article}{
   author={Kruglyak, S. A.},
   author={Rabanovich, V. I.},
   author={Samo\u\i lenko, Yu. S.},
   title={On sums of projections},
   language={Russian, with Russian summary},
   journal={Funktsional. Anal. i Prilozhen.},
   volume={36},
   date={2002},
   number={3},
   pages={20--35, 96},
%   issn={0374-1990},
   translation={
      journal={Funct. Anal. Appl.},
      volume={36},
      date={2002},
      number={3},
      pages={182--195},
%      issn={0016-2663},
   },
%   review={\MR{1935900}},
%   doi={10.1023/A:1020193804109},
}

% \bib{LR}{article}{
%    author={Lackey, Brad},
%    author={Rodrigues, Nishant},
%    title={Nonlocal games, synchronous correlations, and Bell inequalities},
%    eprint={arXiv:1707.06200},
% }

\bib{Ro}{article}{
   author={Man\v cinska, Laura},
   author={Roberson, David E.},
   title={Quantum homomorphisms},
   journal={J. Combin. Theory Ser. B},
   volume={118},
   date={2016},
   pages={228--267},
%   issn={0095-8956},
%   review={\MR{3471851}},
%   doi={10.1016/j.jctb.2015.12.009},
}

\bib{NGHA}{article}{
  author={Navascu{\'e}s, Miguel},
  author={Guryanova, Yelena},
  author={Hoban, Matty~J.},
  author={Ac{\'i}n, Antonio},
  title={Almost quantum correlations},
  journal={Nat. Commun.},
  volume={6},
  date={2015},
  pages={6288},
}

\bib{NPA}{article}{
  author={Navascu{\'e}s, Miguel},
  author={Pironio, Stefano},
  author={Ac{\'i}n, Antonio},
  title={A convergent hierarchy of semidefinite programs characterizing the set of quantum correlations},
  journal={New J. Phys.},
  volume={10},
  date={2008},
  number={7},
  pages={073013},
}

\bib{Oz13}{article}{
   author={Ozawa, Narutaka},
   title={About the Connes embedding conjecture: algebraic approaches},
   journal={Jpn. J. Math.},
   volume={8},
   date={2013},
%   number={1},
   pages={147--183},
%   issn={0289-2316},
%   review={\MR{3067294}},
%   doi={10.1007/s11537-013-1280-5},
}

\bib{PSSTW}{article}{
   author={Paulsen, Vern I.},
   author={Severini, Simone},
   author={Stahlke, Daniel},
   author={Todorov, Ivan G.},
   author={Winter, Andreas},
   title={Estimating quantum chromatic numbers},
   journal={J. Funct. Anal.},
   volume={270},
   date={2016},
%   number={6},
   pages={2188--2222},
%   issn={0022-1236},
%   review={\MR{3460238}},
%   doi={10.1016/j.jfa.2016.01.010},
}
		
\bib{PT}{article}{
   author={Paulsen, Vern I.},
   author={Todorov, Ivan G.},
   title={Quantum chromatic numbers via operator systems},
   journal={Q. J. Math.},
   volume={66},
   date={2015},
%   number={2},
   pages={677--692},
%   issn={0033-5606},
%   review={\MR{3356844}},
%   doi={10.1093/qmath/hav004},
}

\bib{Ro-thesis}{thesis}{
   author={Roberson, David E.},
   title={Variations on a theme: Graph homomorphisms},
   type={Ph.D.  thesis},
   organization={University of Waterloo},
   date={2013},
}

\bib{Sl17}{article}{
   author={Slofstra, William},
   title={The set of quantum correlations is not closed},
   eprint={arxiv:1703.08618},
   status={preprint},
}

\bib{Ts1}{article}{
   author={Tsirelson, B. S.},
   title={Some results and problems on quantum Bell-type inequalities},
   journal={Hadronic J. Suppl.},
   volume={8},
   date={1993},
%   number={4},
   pages={329--345},
%   issn={0882-5394},
%   review={\MR{1254597}},
}

\bib{Ts2}{article}{
   author={Tsirelson, B. S.},
   title={Bell inequalities and operator algebras},
   note={Problem statement for website of open problems at TU Braunschweig},
   date={2006},
   eprint={http://web.archive.org/web/20090414083019/http://www.imaph.tu-bs.de/qi/problems/33.html},
}

\end{thebibliography}
\end{document}